\documentclass{amsproc}
\usepackage{eurosym}
\usepackage{amssymb}
\usepackage{amsfonts}
\usepackage[utf8]{inputenc}
\usepackage[OT4]{fontenc}

\newcommand{\field}[1]{\mathbb{#1}}
\newcommand{\C}{\field{C}}
\newcommand{\K}{\field{K}}

\theoremstyle{plain}
\numberwithin{equation}{section}

\newtheorem{theorem}{Theorem}[section]

\newtheorem{lemma}[theorem]{Lemma}

\newtheorem{question}[theorem]{Question}
\newtheorem{conjecture}[theorem]{Conjecture}

\newtheorem{proposition}[theorem]{Proposition}
\newtheorem{corollary}[theorem]{Corollary}

\DeclareMathOperator{\graph}{graph}

\def\p{\mathbb{P}}

\begin{document}
\title[]{Quantitative properties of the non-properness set of a polynomial map, a positive characteristic case}

\author[Zbigniew Jelonek]{Zbigniew Jelonek} 
\author[Micha{\l} Laso\'n]{Micha{\l} Laso\'n}

\address[Z. Jelonek]{Institute of Mathematics of the Polish Academy of Sciences, ul.\'{S}niadeckich 8, 00-656 Warszawa, Poland}
\email{najelone@cyf-kr.edu.pl}
\address[M. Laso\'n]{Institute of Mathematics of the Polish Academy of Sciences, ul.\'{S}niadeckich 8, 00-656 Warszawa, Poland}
\email{michalason@gmail.com}

\thanks{Research supported by Polish National Science Centre grant no. 2013/09/B/ST1/04162.}

\begin{abstract}
Let $f:\K^n\rightarrow\K^m$ be a generically finite polynomial map of degree $d$ between affine spaces. In \cite{jela2} we proved that if $\K$ is the field of complex or real numbers, then the set $S_f$ of points at which $f$ is not proper, is covered by polynomial curves of degree at most $d-1$. In this paper we generalize this result to positive characteristic. We provide a geometric proof of an upper bound by $d$.
\end{abstract}

\maketitle

\section{Introduction}

We begin by recalling some necessary definitions, notions and facts (cf. \cite{jela2}). Unless stated otherwise, $\K$ is an arbitrary algebraically closed field. All affine varieties are considered to be embedded in an affine space. 

An irreducible affine curve $\Gamma\subset\K^m$ is called a \emph{parametric curve of degree at most $d$}, if there exists a non-constant polynomial map $f:\K\rightarrow\Gamma$ of degree at most $d$ (by degree of $f=(f_1,\dots,f_m)$ we mean $\max_i \deg f_i$). A curve is called \emph{parametric} if it is parametric of some degree. 

\begin{proposition}[Proposition 1.2 \cite{jela2}, cf. Proposition 2.4 \cite{jela}]\label{k-uniruledprop}
Let $X\subset\K^m$ be an irreducible affine variety of dimension $n$, and let $d$ be a constant. The following conditions  are equivalent:
\begin{enumerate}
\item for every point $x\in X$ there exists a parametric curve $l_x\subset X$ of degree at most $d$ passing through $x$,
\item there exists an open, non-empty subset $U\subset X$, such that for every point $x\in U$ there exists a parametric curve $l_x\subset X$ of degree at most $d$ passing through $x$,
\item there exists an affine variety $W$ of dimension $\dim X-1$, and a dominant polynomial map $\phi:\K\times W\ni (t,w)\to \phi(t,w)\in X$ such that $\deg_t \phi \leq d$.
\end{enumerate}
\end{proposition}

We say that an affine variety $X$ has \emph{degree of $\K$-uniruledness at most $d$} if all its irreducible components satisfy the conditions of Proposition \ref{k-uniruledprop} (cf. Definition 1.3 \cite{jela2}). An affine variety is called \emph{$\K$-uniruled} if it has some degree of $\K$-uniruledness. To simplify the notion we assume that the empty set has degree of $\K$-uniruledness equal to zero, in particular it is $\K$-uniruled.

Let $f:X\rightarrow Y$ be a generically finite polynomial map between affine varieties. Recall that $f$ is \emph{finite (proper)} at a point $y\in Y$, if there exists an open neighborhood $U$ of $y$ such that $f\vert_{f^{-1}(U)}:f^{-1}(U)\rightarrow U$ is a finite map. The set of points at which $f$ is not finite (proper) we denote by $S_f$ (see \cite{jel}, and also \cite{jel1,jel2,stas}). The set $S_f$ is a good measure of non-properness of the map $f$, moreover it has interesting applications \cite{jel3,jel4,saf,ha}.

\begin{theorem}[Theorem 4.1 \cite{jel5}]\label{hiper}
Let $f:X\rightarrow Y$ be a generically finite polynomial map between affine varieties. The set $S_f$ is a hypersurface in $\overline{f(X)}$ or it is empty. Additionally, if $X$ is $\K$-uniruled, then the set $S_f$ is also $\K$-uniruled.
\end{theorem}

\begin{theorem}[Theorem 4.6 \cite{jel0}]\label{hiperc}
Let $f:X\to\K^n$ be a dominant, generically finite, separable, polynomial map between affine varieties. Then the set $S_f$ is a hypersurface of  degree at most
$$\frac{\deg X\prod_i \deg  f_i-\mu(f)}{\emph{min}_i\deg f_i},$$
where $\mu(f)$ is the multiplicity of $f$, or it is empty.
\end{theorem}

Due to Theorem \ref{hiper}, if $X$ is a $\K$-uniruled affine variety and $f:X\rightarrow \K^m$ is a generically finite polynomial map, then the set $S_f$ is $\K$-uniruled. In \cite{jela2} we studied the behavior of degree of $\K$-uniruledness of the set $S_f$ as a function of degree of the map $f$.

\begin{theorem}[Theorem 3.2 \cite{jela2}]\label{heorem1}
Let $f:\C^n\rightarrow\C^m$ be a generically finite polynomial map of degree $d$. Then the set $S_f$ has degree of $\C$-uniruledness at most $d-1$.
\end{theorem}

\begin{theorem}[Theorem 3.5 \cite{jela2}]\label{heorem2}
Suppose $X$ is an affine variety with degree of $\C$-uniruledness at most $d_1$. Let $f:X\rightarrow\C^m$ be a generically finite polynomial map of degree $d_2$. Then the set $S_f$ has degree of $\C$-uniruledness at most $d_1d_2$.
\end{theorem} 

In \cite{jela2} we provide examples showing that the above bounds are tight. Their proofs intensively use the topology of $\C$, thus they can not be adapted for an arbitrary algebraically closed field. However, by the Lefschetz principle, both statements are true for an arbitrary algebraically closed field $\K$ of characteristic zero.

In this paper we deal with the positive characteristic case. In Theorem \ref{kn} we almost (up to additive constant $1$) generalize Theorem \ref{heorem1}. That is, we prove that the set $S_f$ has degree of $\K$-uniruledness at most $d$. Our proof uses cute and simple geometric idea.

\section{Main result}

\begin{proposition}[\cite{je00}]\label{graph}
Let $f:X\rightarrow Y$ be a generically finite map between affine varieties $X\subset\K^n$ and $Y\subset\K^m$. Let 
$$\graph(f)=\{(x,y)\in X\times Y:y_i=f_i(x)\},$$ 
and let $\overline{\graph(f)}$ be its closure in $\p^n\times\K^m$. Then there is an equality $$S_f=\pi(\overline{\graph(f)}\setminus\graph(f)),$$ 
where $\pi$ denotes the projection to the second factor $\p^n\times\K^m\rightarrow\K^m$.
\end{proposition}

\begin{lemma}\label{tool}
Let $A\subset\K^n$ be an affine set, and let $f$ be a regular function on $\K^n$ not equal to
$0$ on any component of $A$. Suppose that for each $c\in \K\setminus\{0\}$ the set 
$$A_c:=A\cap\{x\in\K^n:f(x)=c\}$$
has degree of $\K$-uniruledness at most $d$. Then the set $A_0$ also has degree of $\K$-uniruledness at most $d$.
\end{lemma}
\begin{proof}
Suppose affine set is given by $A=\{x\in\K^n:g_1(x)=\dots=g_r(x)=0\}$. For $a=(a_1,\dots,a_n)\in\K^n$ and $b=(b_{1,1}:\dots:b_{d,n})\in\p^{dn-1}$, let 
$$\varphi_{a,b}:\K\ni t\rightarrow (a_1+b_{1,1}t+\dots+b_{1,d}^dt^d,\dots,a_n+b_{n,1}t+\dots+b_{n,d}^dt^d)\in\K^n$$ 
be a parametric curve of degree at most $d$. Let us consider a variety and a projection
$$\K^n\times\p^{dn-1}\supset V=\{(a,b)\in\K^n\times\p^{dn-1}:\forall_{t,i}\;g_i(\varphi_{a,b}(t))=0\;\text{and}$$
$$\forall_{t_1,t_2}\;f(\varphi_{a,b}(t_1))=f(\varphi_{a,b}(t_2))\}\ni(a,b)\rightarrow a\in\K^n.$$
The definition of the set $V$ says that parametric curves $\varphi_{a,b}$ are contained in $A$ and $f$ is constant on them. Hence $V$ is closed and the image of the projection is contained in $A$. Moreover the image contains every $A_c$ for $c\in \K\setminus\{0\}$, since they are filled with
parametric curves of degree at most $d$. But since $\p^{dn-1}$ is complete and $V$ is closed, the image of the projection is closed. Hence it must be the whole set $A$. In particular $A_0$ is contained in the image, so it is filled with parametric curves of degree at most $d$.
\end{proof}

\begin{theorem}\label{kn}
Let $\K$ be an arbitrary algebraically closed field. If $f:\K^n\rightarrow\K^m$ is a generically finite polynomial map of degree $d$,  then the set $S_f$ has degree of $\K$-uniruledness at most $d$.
\end{theorem}

\begin{proof}
If $n=1$, then the map is proper and $S_f$ is empty. Suppose $n\geq 2$. Due to Proposition \ref{graph} $$S_f=\pi(\overline{\graph(f)}\setminus\graph(f)).$$

It is enough to prove that the set $\overline{\graph(f)}\setminus\graph(f)$ is filled with parametric curves of degree at most $d$. Indeed, we can take images of these curves under projection $\pi$. Projection has degree $1$, so images of these curves are either parametric curves of degree at most $d$, or points. Since map $f$, and as a consequence also $\pi$, is generically finite, only on a codimension $1$ subvariety images of curves will become points. This by Proposition \ref{k-uniruledprop} is acceptable.

Let us denote coordinates in $\p^n\times\K^m$ by $(x_{0} : \dots : x_n;x_{n+1},\dots,x_{n+m})$. Take an arbitrary point $$z\in\overline{\graph(f)}\setminus\graph(f)=\overline{\graph(f)}\cap\{x_0=0\}.$$ 
We are going to show that through $z$ passes a parametric curve of degree at most $d$. Since $z_0=0$, there exists $1\leq i\leq n$ such that $z_i\neq 0$. Consider an affine set $A:=\overline{\graph(f)}\cap\{x_i\neq 0\}$ and a function $f=\frac{x_0}{x_i}$ regular on it.  Consider sets
$$A_c:=A\cap\{x\in\K^n:f(x)=c\}=\overline{\graph(f)}\cap\{x_i\neq 0\}\cap\{x_0=cx_i\}.$$ 
For $c\neq 0$ sets $A_c$ are filled with parametric curves of degree at most $d$. Indeed, we can take an index $j$ distinct from $0,i$ and consider curves
$$\K\ni t\rightarrow (c:a_1:\dots:a_{i-1}:1:\dots:a_{j-1}:t:a_{j+1}:\dots)\xrightarrow{f} A_c\subset\graph(f).$$
Hence, by Lemma \ref{tool}, the set 
$$A_0=\overline{\graph(f)}\cap\{x_i\neq 0\}\cap\{x_0=0\}$$ 
is also filled with such curves. In particular we get that though $z$ passes a parametric curve of degree at most $d$ which is contained in $$\overline{\graph(f)}\setminus\graph(f)=\overline{\graph(f)}\cap\{x_0=0\}.$$ 
This finishes the proof.
\end{proof}

By looking carefully at the above proof we get the following slightly more general result.

\begin{corollary}\label{kkw}
Let $\K$ be an arbitrary algebraically closed field, and let $W$ be an affine variety. If $f:\K^2\times W\rightarrow\K^m$ is a generically finite polynomial map of degree $d$, then the set $S_f$ has degree of $\K$-uniruledness at most $d$.
\end{corollary}

If we were able to prove the assertion of the corollary for maps $f:\K\times W\rightarrow\K^m$, this would imply Theorem \ref{heorem2} for arbitrary algebraically closed fields (see the proof of Theorem 3.5 \cite{jela2}).

\section{Remarks}

The gap between characteristic zero (Theorem \ref{heorem1}) and arbitrary characteristic (Theorem \ref{kn}) suggests the following.

\begin{conjecture}
	Let $\K$ be an arbitrary algebraically closed field. If $f:\K^n\rightarrow\K^m$ is a generically finite polynomial map of degree $d$,  then the set $S_f$ has degree of $\K$-uniruledness at most $d-1$.
\end{conjecture}

When we consider all generically finite maps $f:\K^n\rightarrow\K^n$ of degree at most $d$, then by Theorem \ref{hiper} hypersurfaces $S_f$ are all $\K$-uniruled, and by Theorem \ref{hiperc} degrees of this hypersurfaces are bounded. In Theorem \ref{kn} we show that their degree of $\K$-uniruledness is also bounded (by $d$). It is reasonable to ask the following general question.

\begin{question}
Does for every $n$ and $d$ exist a universal constant $D=D(n,d)$, such that evey $\K$-uniruled hypersurface in $\K^n$ of degree at most $d$ has degree of $\K$-uniruledness at most $D(n,d)$?
\end{question}

We ask even a stronger question.

\begin{question}\label{22}
Does for every $n$ and $d$ exist a universal constant $D=D(n,d)$, such that if a hypersurface in $\K^n$ of degree at most $d$ contains a parametric curve passing through $O=(0,\dots,0)$, then it also contains such a curve of degree at most $D(n,d)$?
\end{question}

One can show that positive answer to Question \ref{22} is equivalent (for uncountable field $\K$) to the fact that in the set of all hypersurfaces in $\K^n$ of a bounded degree, the set of $\K$-uniruled hypersurfaces forms a closed subset.



\begin{thebibliography}{99}

\bibitem{jel} Z. Jelonek, The set of points at which the polynomial mapping is not proper, Ann. Polon. Math. 58 (1993), 259-266.

\bibitem{jel1} Z. Jelonek, Testing sets for properness of polynomial mappings, Math. Ann. 315 (1999), 1-35.

\bibitem{je00} Z. Jelonek, Local characterization of algebraic manifolds and characterization of components of the set $S_f$, Ann. Polon. Math. (2000), 7-13.

\bibitem{jel2} Z. Jelonek, Topological characterization of finite mappings, Bull. Acad. Polon. Sci. Math. 49 (2001), 279-283.

\bibitem{jel3} Z. Jelonek, K. Kurdyka, On asymptotic critical values of a complex polynomial,  J. Reine Angew. Math. 565 (2003), 1-11.

\bibitem{jel4} Z. Jelonek, On the effective Nullstellensatz, Invent. Math. 162 (2005), 1-17.

\bibitem{jel0} Z. Jelonek, On the \L ojasiewicz exponent, Hokkaido Math. J. 35 (2006), 471-485.

\bibitem{jel5} Z. Jelonek, On the Russell problem, J. Algebra 324 (2010), no. 12, 3666-3676.

\bibitem{jela} Z. Jelonek, M. Laso\'{n}, The set of fixed points of a unipotent group, J. Algebra 322 (2009), 2180-2185.

\bibitem{jela2}  Z. Jelonek, M. Laso\'{n}, Quantitative properties of the non-properness set of a polynomial map, Manuscripta Math. 156 (2018), no. 3-4, 383-397.

\bibitem{saf} M. Safey El Din, Testing Sign Conditions on a Multivariate Polynomial and Applications, Mathematics in Computer Science 1 (2007), 177-207.

\bibitem{stas} A. Stasica, Geometry of the Jelonek set, J. Pure Appl. Algebra 198 (2005), 317-327. 

\bibitem{ha} H. Vui, P. S\'{o}n, Representations of positive polynomials and optimization on noncompact semialgebraic sets, SIAM J. Optim. 20 (2010), 3082-3103.

\end{thebibliography}
\end{document}